\documentclass[10pt,dvipsnames]{article}
\usepackage[a4paper]{anysize}\marginsize{3.5cm}{3.5cm}{1.3cm}{2cm}
\pdfpagewidth=\paperwidth \pdfpageheight=\paperheight
\usepackage{amsfonts,amssymb,amsthm,amsmath,eucal}
\usepackage[T1]{fontenc}
\usepackage[utf8]{inputenc}
\usepackage{pgf}
\usepackage{array}
\usepackage{tikz} 
\usepackage{float}
\usepackage{subfigure}
\usepackage{caption}
\usetikzlibrary{arrows}
\usepackage{url}
\usepackage{marginnote}
\usepackage{booktabs}
\usepackage{pdflscape}
\usepackage[title]{appendix}

\usepackage{hyperref}
\hypersetup{
    colorlinks,
    citecolor=black,
    filecolor=black,
    linkcolor=black,
    urlcolor=black
}

\theoremstyle{plain}
\newtheorem{thm}{Theorem}[section]

\theoremstyle{definition}
\newtheorem{definition}[thm]{Definition}

\newtheorem{example}[thm]{Example}

\newcommand\C{\mathbb C}

\newcommand{\PP}{\mathbb{P}}
\DeclareMathOperator{\Jac}{Jac}

\newcommand\newop[2]{\def#1{\mathop{#2}\nolimits}}
\newop\Der{Der}
\newop\mult{mult}
\newop\characteristic{char}

\newcommand\parl[2]{\partial_{#2} #1}
\def\keywordname{{\bfseries Keywords}}%
\def\keywords#1{\par\addvspace\medskipamount{\rightskip=0pt plus1cm
\def\and{\ifhmode\unskip\nobreak\fi\ $\cdot$
}\noindent\keywordname\enspace\ignorespaces#1\par}}
\def\subclassname{{\bfseries Mathematics Subject Classification
(2000)}\enspace}
\def\subclass#1{\par\addvspace\medskipamount{\rightskip=0pt plus1cm
\def\and{\ifhmode\unskip\nobreak\fi\ $\cdot$
}\noindent\subclassname\ignorespaces#1\par}}
\begin{document}

\author{Ewelina Nawara}
\title{The Hesse pencil of plane curves and osculating conics}
\date{\today}
\maketitle
\thispagestyle{empty}

\begin{abstract}

In this paper, we revisit the classical problem of determining osculating conics and sextactic points for a given algebraic curve. Our focus is on a particular family of plane cubic curves known as the Hesse pencil. By employing classical tools from projective differential geometry, we derive explicit coordinates for these special points. The resulting formulas not only clarify previous approaches but also lead to the construction of new families of free and nearly free curves, extending recent findings the freeness of curves.

\subclass{14H50, 53A15, 14N05}
\end{abstract}
\section{Introduction}

The study of osculating curves was initiated by Salmon and Cayley in the XIX century, but each of them approached it with a different concept.

Salmon investigated whether, given a fixed cubic curve $C_1$, it is possible to find another cubic curve $C_2$ such that the intersection index at their points of intersection equals nine. The points where this maximum intersection index, occurs were named by Salmon as having nine-point contact. In his research, he determined that such a cubic curve contains $81$ such points, among which there are the $9$ inflection points. The equations defining these points were later detailed and analyzed by Hart in 1875 \cite{Hart}. In 1876, Halphen studied ``coincidence points'' on a plane curve. He considered a family of cubic curves that share an $8$-point contact at a given point $P_1$ with both the reference curve and each other. These cubic curves also intersect at a second point $P_2$. If $P_2$ coincides with $P_1$, then $P_1$ is called a coincidence point of the curve \cite{Halphen}. Both Salmon's and Halphen's studies focused on families of cubic curves intersecting at points with intersection index of $9$. However, this concept differs from that of nine-point contact in the so-called second Hessian context.

Cayley proposed a different approach by considering a given curve of degree at least $3$ and intersecting it with a curve of degree $2$. His studies also focused on the intersection index at the points where these curves meet, but with an intersection index of $6$. He referred to these points as \emph{sextactic points}. In \cite{Cay59}, he defined sextactic points and explained how to identify them. He also presented the equations that a conic must satisfy at these points to achieve an intersection index of $6$. Later, in \cite{Cay65}, he derived the equation of the curve that must pass through sextactic points, inflection points, and singularities, naming this curve \emph{the second Hessian}.

Since Cayley's work, mathematicians have continued to develop the theory of sextactic points. In \cite[Chapter VI, Theorem 17]{Coo}, Coolidge presented a formula for the number of sextactic points on a curve, subject to certain constraints related to the curve's genus, as well as the number of nodes and cusps of both the curve and its dual. A similar study was conducted by Thorbergsson and Umehara \cite{masaaki}, in which they established bounds for the number of sextactic points on a simple closed curve in the real projective plane. Also, some recent work has led to an increase in interest in the topic of the $2$-Hessian, osculating conics and sextactic points, see~\cite{MT,MZ,SzeSzp} for details.

Despite the classical results on sextactic points and the second Hessian, explicit computations for specific and symmetric families of curves, such as the Hesse pencil, are still relatively rare in the literature. The high degree of symmetry and rich geometric structure of this pencil make it a particularly interesting object for explicit study. Understanding how sextactic points and their associated osculating conics behave in this family may also shed new light on the geometry of the second Hessian.

In this article, we consider the Hesse pencil of plane cubic curves, i.e., curves given by the equation  
$$F_{\lambda,\mu}(x,y,z)=\lambda(x^3+y^3+z^3)+\mu xyz,$$
where $(\lambda,\mu)\in\mathbb{P}^1$. Our aim is to present detailed and self-contained computations concerning the second Hessian and the osculating conics at sextactic points within the Hesse pencil. This not only complements the classical theory with concrete examples but also helps to clarify the geometric behavior of these invariants in a highly symmetric setting.

The paper is structured as follows. In Section 2, we recall the necessary definitions, notations, and theorems required to present our results. Section 3 contains the computations needed to derive the equation of the second Hessian and the osculating conics associated with the given curve. Noting that the set of all sextactic points is closed under the action of a certain group, it suffices to compute the equation of an osculating conic for just three points. We present two distinct methods to achieve this.

The simplicity of the resulting conic equations allows us to construct free curves of a special type and to verify whether some results from \cite{dimca} also hold for special members of the Hesse pencil, as well as for a general member. As a consequence, we identify three additional cases in which free and nearly free curves can be constructed using the equation of a cubic and its osculating conics. These results are summarized in Theorem~\ref{thm:main}. These constructions contribute to the understanding of the geometric structure of curves related to the Hesse pencil.

\section{Preliminaries} 

We begin this section by recalling definitions of the second Hessian, osculating conics and sextactic points. In order to do that, we need the following notation.

Let $S$ denote a ring $\C[x,y,z]$ of three variables over complex numbers $\C$. For given polynomials $f,g,h \in S$ we denote by $\Jac(f,g,h)$  the determinant of the matrix $M_J$, i.e.
  $$ \Jac(f,g,h):=\det\begin{pmatrix}
        \partial_x f & \partial_y  f& \partial_z f\\
        \partial_x g & \partial_y g & \partial_z g\\
        \partial_x h & \partial_y h & \partial_z h 
    \end{pmatrix}=\det M_J,$$
and by $H$ we denote Hessian of the polynomial $F \in S$, namely 
   $$ H:=\Jac(F_x,F_y,F_z)=\det\begin{pmatrix}
        \partial^2_{x^2} F & \partial^2_{xy} F & \partial^2_{xz} F \\
        \partial^2_{yx} F & \partial_{y^2}^2 F& \partial_{yz}^2 F \\
        \partial_{zx}^2 F & \partial_{zy}^2 F& \partial_{z^2}^2 F    \end{pmatrix}.$$
Denote by $M_{i,j}$ the minor of a matrix $M_J$ of $i$ row and $j$ column. Consider the following two vectors: $M$, which consists of some $2 \times 2$ minors of $M_J$, i.e. $M=[M_{11},M_{22},M_{33},M_{32},M_{31},M_{21}]$, and $V_H$ which contains the second derivatives of Hessian $H$, i.e.
$$V_H=[\partial^2_{x^2}H, \;\partial^2_{y^2}H,\;\partial^2_{z^2}H,\;2\partial^2_{yz}H,\;2\partial^2_{xz}H,\;2\partial^2_{xy}H].$$
In what follows, taking the partial derivative $\partial_{\bullet}$ from a vector means calculating a partial derivative $\partial_{\bullet}$ from each entry of the vector. We introduce the notation $$\partial_{\bullet}(\Omega_1(M,V_H)):=\partial_{\bullet} M \circ V_H \;\text{ and } \;\partial_{\bullet}(\Omega_2(M,V_H)):= M \circ \partial_{\bullet} V_H,$$
which denotes the standard scalar product of two vectors with appropriate derivation.
Additionally, polynomial $\Psi$ denotes the determinant of the so-called bordered Hessian calculated for polynomials $H$ and $F$, which is 
\begin{align*}
\Psi:=-\det\begin{pmatrix}  
        0 & \partial_x H & \partial_y H & \partial_z H \\
       \partial_x H & \partial^2_{x^2}  F  & \partial^2_{xy} F & \partial^2_{xz} F \\
          \partial_y H & \partial^2_{yx} F & \partial^2_{y^2} F & \partial^2_{yz} F  \\
        \partial_z H &\partial^2_{zx} F & \partial^2_{zy} F & \partial_{z^2}^2 F   \\
    \end{pmatrix},    
\end{align*}
or, equivalently, after Laplace expansion
\begin{footnotesize}
    \begin{align} \label{eq:psi}
    \Psi:=[M_{11},M_{22},M_{33},M_{32},M_{31},M_{21}] \circ [(\partial_xH)^2,(\partial_yH)^2,(\partial_zH)^2,2 \partial_yH\partial_zH,2\partial_xH\partial_zH,2\partial_xH\partial_yH]
    \end{align}
\end{footnotesize}

Let  $P \in \PP^2_{\C}$. Following Cayley, for a fixed polynomial $F$ of degree $d$, we denote
\begin{equation}
\label{eq:DFD2F}
    \begin{aligned}
    DF_P(x,y,z)=&x\parl{F(P)}{x}+y\parl{F(P)}{y}+z\parl{F(P)}{z},\\
    D^2F_P(x,y,z)=&x^2\parl{^2F(P)}{x}+y^2\parl{^2F(P)}{y}+z^2\parl{^2F(P)}{z}
    \\&+2xy\parl{\parl{F(P)}{y}}{x}+2xz\parl{\parl{F(P)}{z}}{x}+2yz\parl{\parl{F(P)}{z}}{y}.
    \end{aligned}
\end{equation}
If the point $P$ belongs to the set of zeros of $F$, then, in fact, $DF_P=0$ defines a polar curve of rank $d-1$, or equivalently, the tangent line to $F=0$ at $P$. Analogously, $D^2F = 0 $ can be viewed as a polar curve of rank $d-2$ (see e.g. \cite[Equation 5.21]{Belta}).

The concept of sextactic points and osculating conics is a generalization of the idea of a tangent line to an inflection point.
An inflection point on a curve is a smooth point where the tangent intersects the curve with multiplicity higher than expected, while a sextactic point is a point where the osculating conic intersects the curve with multiplicity higher than expected. 
Cayley was interested in finding osculating conics, and as a result of this research he obtained the following theorem  \cite[Section 15]{Cay59}.

\begin{thm}[Osculating conics]
\label{Osculating}
Let $C$ be a plane curve of degree $d$ given by a polynomial $F$. If $P$ is a point on $C$ that is neither singular nor an inflection
point and $$9H^3\Lambda=-3 (M \circ V_H) H +4 \Psi,$$ then the conic $O_P$ given by the equation
$$O_P:D^2F_P-\left(\frac{2}{3H(P)} DH_P+\Lambda(P)DF_P\right)DF_P = 0,$$ intersects $C$ at $P$ with multiplicity at least $5$.
\end{thm}
The conic $O_P$, defined in the previous theorem, is called the osculating conic of $C$ at the point $P$. It occurred that the number of points $P$, for which we can obtain osculating conic, is determined by number $d(12d-27)$ (compare \cite[Section 27]{Cay65}). This justifies to give a name for such points.

\begin{definition}[Sextactic point] 
Let $P$ be a smooth point on a curve $C$ that is not an inflection
point. Then $P$ is called a sextactic point if intersection index
$$(C.O_P)_P \geq 6,$$
where $O_P$ is the osculating conic of $C$ at $P$. A sextactic point $P$ is said to be of
type $s$, or $s$-sextactic, when $$ s=(O_P.C)_P - 5.$$
\end{definition}

While in \cite{Cay59} Cayley was interested in defining the equation of osculating conic for a given point $P$, a few years later he focused on finding all points for which $O_P$ exists. As a result, he obtained a polynomial, the so-called second Hessian, which passes through all such points \cite[Section 27]{Cay65}. What turned out, his formula for second Hessian has mistake and was recently corrected by Maugesten and Moe in 2018 (see \cite[Theorem 1.1]{MaMo19}).

\begin{definition}[The second Hessian]
Let $C: F=0$ be a plane curve of degree $d$. We define its second Hessian $H_2(F)$ as follows
    \begin{align*}
        H_2(F):=&(12d^2 -54d+57)H \Jac(F,H,\Omega_1(M,V_H))\\
    &+(d-2)(12d-27)H
    \Jac(F,H,\Omega_2(M,V_H)) \\
    &-20(d-2)^2\Jac(F,H,\Psi).
    \end{align*}
\end{definition}

\section{Sextactic points and osculating conics for Hesse pencil curves}

In this section, we present our main result by considering the so-called \emph{Hesse pencil}, defined by
$$F_{\lambda,\mu}(x,y,z)=\lambda(x^3+y^3+z^3)+6\mu xyz,$$
where $(\lambda,\mu) \in \mathbb{P}^1$. For a comprehensive discussion on the Hesse pencil, the reader is encouraged to consult \cite{dolga}.

In this paper, we provide all necessary calculations required to determine the coordinates of the sextactic points and the equation of the osculating conics. If $\lambda = 0$ and $\mu \neq 0$, then $F_{0, \mu} = xyz$, which is not relevant for our analysis. The other boundary case, $F_{\lambda, 0} = x^3 + y^3 + z^3$, has been addressed in \cite{SzeSzp}. Therefore, we assume $\lambda \neq 0$ and $\mu \neq 0$, and we rewrite the equation of $F_{\lambda, \mu}$ using $t := \frac{6\mu}{\lambda} \neq 0$. Hence, from this point onward, we consider
\begin{align*}
    F_t(x,y,z)=x^3+y^3+z^3+txyz,
\end{align*}
and omit the subscript in the notation of $F_t$.

Assume that $P=(a,b,c) \in \PP^2$ is a point on the curve $C : F=0$. We start with finding the equation of the second Hessian. In order to do it, we calculate
\begin{align}
    \left[\partial_x F,\; \partial_y F,\; \partial_z F\right]&=\left[3x^2+tyz,\; 3y^2+txz,\; 3z^2+tyx \right], \nonumber \\
    \left[\partial^2_{x^2} F,\; \partial^2_{y^2} F,\; \partial^2_{z^2} F,\; \partial^2_{xy} F,\; \partial^2_{xz} F,\; \partial^2_{yz} F \right]&=\left[6x,\; 6y,\; 6z,\; tz,\; ty,\; tx\right], \label{eq:SecondDer}
\end{align} which we use to find the following equations
\begin{align*}
DF_P(x,y,z)&:=x(3a^2+tbc)+y(3b^2+tac)+z(3c^2+tab),  \\
D^2F_P(x,y,z)&:=6(ax^2+by^2+cz^2)+2t(cxy+bxz+ayz), 
\end{align*}
By \eqref{eq:SecondDer} and the definition of Hessian
\begin{align}
\label{eq:Mhessian}
   H(x,y,z):=\det\begin{pmatrix}
        6x & tz & ty\\
        tz & 6y & tx\\
        ty & tx & 6z 
    \end{pmatrix}=(216+2t^3) xyz -6t^2(x^3+y^3+z^3),
\end{align} and, since during computations we are interested in points lying on $F$, we sometimes use the following form of $H$, which follows from substituting the equation of $F$ into $H = 0$, and the value of $H(P)$, namely
 \begin{align}
 \label{eq:hessian}
H(x,y,z)=8xyz(27+t^3),\quad H(P)=8abc(27+t^3).
 \end{align}
We need to determine all singular points, which by definition, cannot be sextactic points. Solving the system of equations $F = H = 0$, we find that the condition for the point $P$ to not be singular is equivalent to $t^3 + 27 \neq 0$, which we assume from now on.  
From \eqref{eq:Mhessian} we have the following vectors
\begin{align} 
V_H =& \left[-36t^2x,\, -36t^2y,\, -36t^2z,\,  (432+4t^3)x,\, (432+4t^3)y,\, (432+4t^3)z\right] \label{VH} \\
M =& \left[36yz-t^2x^2,\, 36xz-t^2y^2,\, 36xy-t^2z^2,\, -6tx^2+t^2yz,\, -6ty^2+t^2xz,\, -6tz^2+t^2xy\right] \nonumber
\end{align}
and from \eqref{eq:DFD2F} we derive
\begin{footnotesize}
    \begin{equation*}
    DH_P(x,y,z):=x\left((2t^3+216)bc-18t^2a^2\right)+y\left((2t^3+216)ac-18t^2b^2\right)+z\left((2t^3+216)ab-18t^2c^2\right).
    \end{equation*}
\end{footnotesize}
In the following computations we can use the vector $V_H$ obtained from the equation of $H$ in \eqref{eq:hessian}, as we are interested in the points of the intersection of second Hessian and $F=0$. Thus
$$
V_H=\left[0,0,0,16x(27+t^3),16y(27+t^3),16z(27+t^3)\right],
$$ from which we obtain
\begin{align*}
    \partial x (\Omega_1(M,V_H))=32t(27+t^3)(-6x^2+tyz), & \quad & \partial x(\Omega_2(M,V_H))=16t(27+t^3)(-6x^2+tyz),\\
    \partial y(\Omega_1(M,V_H))= 32t(27+t^3)(-6y^2+txz), & \quad & \partial y(\Omega_2(M,V_H))=16t(27+t^3)(-6y^2+txz), \\
    \partial z(\Omega_1(M,V_H))=32t(27+t^3)(-6z^2+txy), & \quad & \partial z(\Omega_2(M,V_H))=16t(27+t^3)(-6z^2+txy).
\end{align*}
Now, observe that for any partial derivative $\partial_{\bullet}$, we have two relations
\begin{align*}
    64t(27+t^3)\cdot\partial_{\bullet}F -12t^2\cdot\partial_{\bullet}H+\partial_{\bullet}(\Omega_1(M,V_H))=0,\\
    32t(27+t^3)\cdot\partial_{\bullet}F -6t^2\cdot\partial_{\bullet}H+\partial_{\bullet}(\Omega_2(M,V_H))=0,
\end{align*}
which give
$\Jac(F,H,\Omega_1(M,V_H))=\Jac(F,H,\Omega_2(M,V_H))=0$. Finally, by \eqref{eq:psi}, the polynomial $\Psi(x,y,z)$ is of the form 
\begin{align*}
\Psi(x,y,z)=&27t^6(x^6+y^6+z^6)-6(23t^6+1728t^3+23328)(x^3y^3+x^3z^3+y^3z^3)+\\
&2t(5t^6+270t^3+23328)xyz(x^3+y^3+z^3)
-t^2(t^6-1404t^3-58320)x^2y^2z^2\\
\stackrel{F=0}{=}& 
192(27+t^3)^2(12(x^3y^3+x^3z^3+y^3z^3)-t^2x^2y^2z^2),
\end{align*}
where symbol $\stackrel{F=0}{=}$ stands for reduction by the equation of $F$. Thus, to derive the equation of $H_2(F)$, we consider only $\Jac(F, H, \Psi)$, which is the determinant of the matrix

\begin{footnotesize}
\begin{equation*}
    \begin{pmatrix}
      3x^2+tyz & 3y^2+txz & 3z^2+tyx \\
        yz & xz & xy \\
        36(x^2y^3+x^2z^3)+10t^2xy^2z^2 & 36(x^3y^2+y^2z^3)+10t^2x^2yz^2 & 36(z^2x^3+z^2y^3)+10t^2x^2y^2z 
        \end{pmatrix},
\end{equation*}
\end{footnotesize}

\noindent
multiplied by the constant factor of $1536(27+t^3)^3$. Therefore, the equation of $H_2(F)$, ignoring constant factors, is 
$$H_2(F)(x,y,z) =(x^3-y^3)(x^3-z^3)(y^3-z^3).$$
This result is not new. It was obtained by Cayley in \cite[Sections 30-39]{Cay59}.

The next stage of this article is to present the exact equation of an osculating conic in a given point $P=(a,b,c) \in C$. This equation was also presented by Calyley, see \cite[Section 34]{Cay65}. Since the formula for the Hessian provided by Cayley is incorrect, specifically, the equation 
$-(1+8l^3)xyz$ given by Cayley (see \cite[p. 387]{Cay59}) should actually be 
$216(1+8l^3)xyz$, the resulting equation for the osculating conic is also incorrect. Our goal here is not only to present the correct formula for such conics but also to simplify the equation by applying certain observations that we introduce later in the article. In this context, we propose two distinct approaches. The first approach refers to Theorem \ref{Osculating} and we provide correct formula based on this theorem. The second approach requires a specific form of point coordinates that we have, along with certain projective changes of coordinates. Although both methods yield the same result, we present also the second approach because it leads to a final equation that is simpler compared to the first approach.

We now proceed with the first approach. To derive the equation of the osculating conic from Theorem \ref{Osculating}, it is necessary to introduce first the formula for $\Lambda$ at a given point $P = (a, b, c) \in C$. It can be observed that by taking $M$ and $V_H$ from \eqref{VH}, we have
$$
M \circ V_H = 12t(t^3-216)(x^3+y^3+z^3+txyz),
$$ which equals to zero for any $P$.
Since 
$$
\frac{4\Psi}{9H^3}=\frac{12(x^3y^3+x^3z^3+y^3z^3)-t^2x^2y^2z^2}{6(t^3+27)x^3y^3z^3},
$$
then
\begin{equation*}
    \Lambda(P)=\frac{4\Psi(P)}{9H^3(P)}=\frac{1}{t^3+27}\left(\frac{2}{c^3}+\frac{2}{b^3}+\frac{2}{a^3}-\frac{t^2}{6abc}\right).
\end{equation*}
After applying the coordinates of point $P$ into $O_P$, we get the equation
\begin{footnotesize}
\begin{align}
\label{eq:OPTheorem}
O_P(x,y,z):=&6(ax^2+by^2+cz^2)+2t(cxy+bxz+ayz)+ \nonumber\\
    &-\frac{1}{t^3+27}\left(\left(\frac{t^3+108}{6a}-\frac{3t^2a}{2bc}
    \right)x+\left(\frac{t^3+108}{6b}-\frac{3t^2b}{2ac}\right)y+\left(\frac{t^3+108}{6c}-\frac{3t^2c}{2ab}\right)z\right)\cdot \nonumber\\
&\cdot\left(x(3a^2+tbc)+y(3b^2+tac)+z(3c^2+tab)\right)+ \\
    &-\frac{1}{t^3+27}\left(\frac{2}{c^3}+\frac{2}{b^3}+\frac{2}{a^3}-\frac{t^2}{6abc}\right)\left(x(3a^2+tbc)+y(3b^2+tac)+z(3c^2+tab)\right)^2. \nonumber
    \end{align}
\end{footnotesize}

The second approach to finding osculating conics is based on the observation that sextactic points of Hesse Pencil curves form orbits with respect to a group 
$G$, whose generators we describe later. To find an equation for 
$O_P$ at a given sextactic point $P$, we first need to determine its coordinates. In order to do it, consider the following system of constraints
\begin{equation*}
\left\{
    \begin{array}{r l}
        H_2(F)&=0 \\
        F&=0  \\
        H &\neq 0
    \end{array}
    \right.
\Longleftrightarrow
    \left\{
    \begin{array}{r l}
        (x^3-y^3)(x^3-z^3)(y^3-z^3) & = 0 \\
        x^3+y^3+z^3+txyz & = 0 \\
        8(t^3+27)xyz & \neq 0
    \end{array}
    \right. .
\end{equation*}
This system is equivalent to nine systems. Due to the symmetry, it is enough to solve the following systems
\begin{equation*}
    \left\{
    \begin{array}{r l}
        y- \epsilon^{2k}x& =0 \\
        x^3+y^3+z^3+txyz & = 0 \\
        xyz & \neq 0
    \end{array}
    \right. ,
\end{equation*}
where $k\in\{0,1,2\}$ and $\epsilon =\frac{1+i\sqrt{3}}{2}$. 
Substituting the first equation to the second, we obtain
\begin{equation}
\label{eq:zxEquation}
    2x^3+\epsilon^{2k}tx^2z+z^3 =0.
\end{equation}
Since $x$ cannot be equal to $0$, without loss of generality, we may substitute $x=1$ in \eqref{eq:zxEquation} and thus consider the equation
\begin{equation}
\label{eq:relation}
    z^3+\epsilon^{2k}tz+2=0.
\end{equation}
Three solutions of these equations are
$$ z_1(\epsilon^{2k}t)=a(t)-\frac{\epsilon^{2k}t}{3a(t)}, \quad z_2(\epsilon^{2k}t)=\frac{\epsilon^{2k+1} t}{3 a(t)}-\epsilon^{5}a(t), \quad z_3(\epsilon^{2k}t)=\frac{\epsilon^{2k+5} t}{3 a(t)}-\epsilon a(t),$$ 
where
$$a(\epsilon^{2k}t)=\frac{\sqrt[3]{3}\cdot\sqrt[3]{\sqrt{3}\sqrt{(\epsilon^{2k}t)^3+27}-9}}{3}=\frac{\sqrt[3]{3}\cdot\sqrt[3]{\sqrt{3}\sqrt{t^3+27}-9}}{3} =a(t).$$

Thus, we obtain 9 points, which coordinates are:

\begin{table}[H]
\centering
\begin{tabular}{lll}
$P_1=(1,1,z_1(t))$, & $P_2=(1,\epsilon^2,z_1(\epsilon^2t))$,   & $P_3=(1,-\epsilon,z_1(-\epsilon t))$,  \\ 
$P_4=(1,1,z_2(t))$, &$P_5=(1,\epsilon^2,z_2(\epsilon^2t))$,  &$P_6=(1,-\epsilon,z_2(-\epsilon t))$,  \\
 $P_7=(1,1,z_3(t))$, & $P_8=(1,\epsilon^2,z_3(\epsilon^2t))$,  & $P_9=(1,-\epsilon,z_3(-\epsilon t))$. \\
\end{tabular}
\end{table}
The remaining 18 points are obtained by permuting the previously determined coordinates points. However, it is important to note that these do not represent all possible permutations of the points, as the omitted permutations correspond to points that coincide with those already listed.
For example, if in $P_2$ we switch the first two coordinates, we get
$(\epsilon^2,1,z_1(\epsilon^2t))$. After dividing all coordinates by $\epsilon^2$, we obtain $$\left(1, \frac{1}{\epsilon^2}, \frac{z_1(\epsilon^2t)}{\epsilon^2}\right)=\left(1,-\epsilon, z_3(-\epsilon t)\right)=P_9.$$
By consolidating all such cases, we arrive at the remaining set of points presented in the list
\begin{table}[H] 
\centering
\begin{tabular}{lll}
$P_{10}=(1,z_1(t),1)$, & $P_{11}=(1,z_1(\epsilon^2 t),\epsilon^2)$,  & $P_{12}=(1,z_1(-\epsilon t),-\epsilon)$, \\
 $P_{13}=(1,z_2(t),1)$,& $P_{14}=(1,z_2(\epsilon^2 t),\epsilon^2)$, & $P_{15}=(1,z_2(-\epsilon t),-\epsilon)$, \\
$P_{16}=(1,z_3(t),1)$, &  $P_{17}=(1,z_3(\epsilon^2 t),\epsilon^2)$, & $P_{18}=(1,z_3(-\epsilon t),-\epsilon)$,  \\
\\ 

 $P_{19}=(z_1(t),1,1)$, & $P_{20}=(z_1(\epsilon^2 t),\epsilon^2,1)$, & $P_{21}=(z_1(-\epsilon t),-\epsilon,1)$, \\
$P_{22}=(z_2(t),1,1)$, & $P_{23}=(z_2(\epsilon^2 t),\epsilon^2,1)$, & $P_{24}=(z_2(-\epsilon t),-\epsilon,1)$, \\
$P_{25}=(z_3(t),1,1)$, & $P_{26}=(z_3(\epsilon^2 t),\epsilon^2,1)$, & $P_{27}=(z_3(-\epsilon t),-\epsilon,1)$. \\ 
\end{tabular}
\end{table}
 These points form orbits with respect to the group $G$ generated by elements
 \begin{align*}
g_0(x,y,z)=(x,z,y), \quad 
 g_1(x,y,z)=(y,z,x), \quad
 g_2(x,y,z)=(x,\epsilon^2 y, -\epsilon z),
 \end{align*}
(compare \cite[Section 4]{dolga}).
Thus, it is enough to find an explicit equations of osculating conic for three of the sextactic points, namely $P_1,P_4, P_7$; and one can use $g_j(x,y,z)$ in order to get all of them. If we apply the formula from \eqref{eq:OPTheorem} to these points, after clearing denominators, we get

\begin{footnotesize}
    \begin{align}
    \label{eq:OPforP1Theorem}
    O_{P}(x,y,z)=&(-3+2z_i t)(-3+12z_i^3-4z_it+z_i^4t-2z_i^2 t^2)\cdot(x^2+y^2) \nonumber\\
    &-(-4z_i^5t^2+z_i^4t^4-15z_i^4t+2z_i^3t^3-90z_i^3+4z_i^2t^2-12z_it-18)\cdot xy\\
    &-(-12z_i^6t+z_i^5t^3-63z_i^5+2z_i^4t^2+z_i^3t^4-9z_i^3t+z_i^2t^3-45z_i^2-2z_it^2-6t)\cdot(xz+yz) \nonumber\\
    &-(-18z_i^7+3z_i^6t^2-12z_i^5t+4z_i^4t^3+45z_i^4-2z_i^3t^2-15z_i^2t-t^2)\cdot z^2, \nonumber
    \end{align}
\end{footnotesize}

\noindent
where $z_i$ in the formula stands for $z_i(t)$, with $i=1,2,3$. If we use relation \eqref{eq:relation} for $z_i(t)$, we can perform further reduction of coefficients in \eqref{eq:OPforP1Theorem}, namely

\begin{footnotesize}
    \begin{align}
    \label{eq:OPforP1Theorem}
    O_{P}(x,y,z)=&3(2z_i^3+7)(z_i^3-1)\cdot(x^2+y^2)+(z_i^6+4z_i^3+22)(z_i^3-1)\cdot xy \nonumber\\
    &+(-6t+15z_i^8-18z_i^5-15z_i^2)\cdot(xz+yz)+(-t^2-z_i^{10}-20z_i^7+40z_i^4-10z_i)\cdot z^2.
    \end{align}
\end{footnotesize}

What is not clearly visible from this representation of the conic equation is that it can be further reformulated with lower-degree coefficients, resulting in a more compact form without losing expressiveness. To achieve this, we refer to \cite[Lemma 2.24]{stud}. Applying this lemma requires performing certain projective changes of coordinates on the equation of $F$, and deriving the equation of the conic $O_P$ from the transformed form of $F$.

We first describe the general outline of the procedure, and subsequently provide the explicit forms of the matrices $A,B,C$ and $D$ used in the transformations. In what follows, we write $z_i$ instead of $z_i(t)$. First, we move the point $(1,1,z_i)$ to $(0,0,1)$ using a projective change of coordinates defined by matrix $A$, resulting in a new equation of $F$, which we denote by $F_1$. Next, we rewrite $F_1$ in the form $F_1(x,y,z) = z^2x + G_1(x,y,z)$ by applying a second projective transformation defined by matrix $B$. Finally, we apply two more projective changes of coordinates, defined by matrices $C$ and $D$, to obtain the desired form $F_1(x,y,z)=z^2x+zy^2+G_2(x,y)$, which, after dehomogenization, serves as the starting point for Lemma 2.24 in \cite{stud}. Before proceeding with the application of this lemma, we provide the explicit forms of all matrices used in the transformations

\begin{align*}
    A=&\begin{pmatrix}
         \frac{1}{z_i} & 0 & 1 \\
         0 & \frac{1}{z_i} & 1 \\
         0 & 0 & z_i \\
     \end{pmatrix}, 
     &B=&\begin{pmatrix}
         \frac{z_i}{3+z_it} & -1 & 0 \\
         0 & 1 & 0 \\
         0 & 0 & 1 \\
     \end{pmatrix},  \\
     C=&\begin{pmatrix}
         1 & 0 & 0 \\
         0 & 1 & 0 \\
         -\frac{3}{2(3+z_it)^2} & -\frac{tz_i-6}{2z_i(3+z_it)} & 1 \\
     \end{pmatrix}, 
     &D=&\begin{pmatrix}
         1 & 0 & 0 \\
         0 & \frac{z_i}{\sqrt{6-tz_i}} & 0 \\
         0 & 0 & 1 \\
     \end{pmatrix}.
 \end{align*}
After dehomogenization, $F_1$ takes the form  
$$
F_1(x,y,1) = x + y^2 + fx^3 + gx^2y + hxy^2 + iy^3,
$$
for which the osculating conic $\widetilde{O_P}$ is given by the equation  
\begin{equation}
\label{eq:correctOP}
\widetilde{O_P}(x,y) = -(i^2 + h)x^2 - ixy + y^2 + x.
\end{equation}
In \eqref{eq:correctOP}, we present the corrected expression for the osculating conic. The version given in \cite[Lemma 2.24]{stud} contains signs error in the coefficient of $x^2$. After substitution for $i$ and $h$, we obtain
\begin{equation*}
\widetilde{O_P}(x,y)= \frac{-9tz_i}{2(3+z_it)^2(6-tz_i)}x^2-\frac{\sqrt{6-tz_i}}{2(3+z_it)}xy+y^2+xz.
\end{equation*}
The equation of $\widetilde{O_P}(x,y)$ after homogenization, and projective change of coordinates, given by
\begin{align}
    \begin{pmatrix}
        x \\
        y \\
        z \\
    \end{pmatrix}=A^{-1}B^{-1}C^{-1}D^{-1}
    \begin{pmatrix}
        x \\
        y \\
        z \\
    \end{pmatrix}=
    \renewcommand{\arraystretch}{1.5}
    \begin{pmatrix}
            3+z_it & 3+z_it & \frac{-2(z_it+3)}{z_i} \\
            0 & \sqrt{6-tz_i} & \frac{-\sqrt{6-tz_i}}{z_i} \\
            \frac{3}{2(z_it+3)} & \frac{tz_i-3}{2(z_it+3)} & \frac{6+z_it}{2z_i(3+z_it)}
    \end{pmatrix}
    \begin{pmatrix}
        x \\
        y \\
        z \\
    \end{pmatrix}
    \renewcommand{\arraystretch}{1.0}
\end{align}
is desired equation of osculating conic in the point $(1,1,z_i)$, i.e.
\begin{footnotesize}
    \begin{equation}
    \label{eq:finalOP}
        O_{P}(x,y,z)=z_i^2(9-6tz_i)\cdot(x^2+y^2)+z_i(15tz_i+18)\cdot(xz+yz)-z_i^2(t^2z_i^2+18)\cdot xy+(z_i^2t^2-18tz_i-36)\cdot z^2.
    \end{equation}
\end{footnotesize}
For the reader’s convenience, we implement a program written in {\tt Singular} \cite{DGPS}, which generates all 27 equations of osculating conics \cite{github}, enabling interested readers to use the given equations in their own computations.

We end this section by presenting an easy example, which shows how to obtain one of the demanded conic.
\begin{example}
    If we want to find the equation of osculating conic in the point $P_{27}$, first observe that $P_{27}=g_2\left(g_1(g_0(P_4))\right)$. Therefore, we need to substitute
    $$x \to z, \;\; y \to \epsilon^2 y, \;\; z \to -\epsilon x,$$
    in the equation \eqref{eq:finalOP}, which leads to
    \begin{footnotesize}
        \begin{equation*}
        O_{P_{27}}(x,y,z)=z_2^2(9-6tz_2)\cdot(z^2-\epsilon y^2)+z_2(15tz_2+18)\cdot(-\epsilon xz+xy)-z_2^2(t^2z_2^2+18)\cdot \epsilon^2 yz+(z_2^2t^2-18tz_2-36)\cdot \epsilon^2 x^2. 
        \end{equation*}
    \end{footnotesize}
\end{example}

\section{Special cases and their role in free curve construction}

The aim of this section is to extend certain results from \cite{dimca}, where the authors investigate the freeness and near-freeness of plane curves defined as the product of the Fermat curve $x^3+y^3+z^3=0$ (i.e., the case $t=0$ in the Hesse pencil) and its osculating conics. We investigate whether analogous phenomena occur for other members of the pencil. Before stating the main result of this section, we recall some basic properties of curves from the Hesse pencil, as presented in \cite{dolga}.

There are two groups of special curves in the Hesse pencil that can be described as zeroes of certain binary forms. Namely, if we consider
$$F_{\lambda,\mu}(x,y,z)=\lambda(x^3+y^3+z^3)+6\mu xyz,$$
then the zeroes of $\mu(\lambda^3-\mu^3)$ define curves from the Hesse pencil that admit an automorphism of order 6 and are called \emph{equianharmonic cubics}. The zeroes of the form $\lambda^6-20\lambda^3\mu^3-8\mu^6$ define another family of curves known as \emph{harmonic cubics}. Since the Fermat curve is one of the equianharmonic cubics, this naturally raises the question of whether the remaining three equianharmonic cubics exhibit the same behavior with respect to freeness. As we shall see, the answer is affirmative.

Let $\mathcal{C}: f=0$ be a reduced curve of degree $d$ in the complex projective plane $\mathbb{P}^2$. The Jacobian ideal of $f$ is defined as $J_f=(\partial_x f, \partial_y f, \partial_z f)$, and the associated graded $S$-module of Jacobian syzygies is given by
$$AR(f)=\left\{(a,b,c)\in S^3: a\cdot\partial_x f+b\cdot\partial_y f+c\cdot\partial_z f=0 \right\}.$$

The curve $\mathcal{C}$  is said to be an $m$-syzygy curve if the $S$-module $AR(f)$ admits a minimal set of $m$ homogeneous generators $r_1,r_2,\dots,r_m$, where each $r_i$ has degree $d_i:=\deg r_i$, and the degrees are arranged in non-decreasing order
$$1\leq d_1 \leq d_2 \leq\ldots\leq d_m.$$ 
The multiset $(d_1,d_2,\dots ,d_m)$ is referred to as the set of exponents associated with the plane curve $\mathcal{C}$.

\begin{definition}
\label{def:freeCurveAlt}
A curve $\mathcal{C}$ that admits exactly two minimal syzygies is called \emph{free}. In this situation, the degrees of the generators satisfy the relation $d_1+d_2=d-1$.
\end{definition}
\begin{definition}
\label{def:nearlyFreeCurveAlt}
A $3$-syzygy curve $\mathcal{C}$ is referred to as \emph{nearly free} if and only if the two largest degrees among the generators coincide, i.e., $d_3=d_2$, and the sum of the two smallest degrees satisfies $d_1+d_2=d$.
\end{definition}

As mentioned earlier, the Fermat curve is one of the equianharmonic cubics in the Hesse pencil. Solving the equation $\mu(\lambda^3-\mu^3)=0$ yields the values of $\mu$ and $\lambda$, and consequently the corresponding values of the parameter $t=\frac{6\mu}{\lambda}$, for the remaining three members of this family. These values are $t\in \{6,6\epsilon^2,6\epsilon^4\}$. We are now ready to state the main theorem of this section.

\begin{thm}
\label{thm:main}
Let $F(x,y,z)=x^3+y^3+z^3+txyz$, and let $t\in\{0,6,6\epsilon^2,6\epsilon^4\}$. Let $C_i$, for $i \in \{1,2,\ldots,27\}$, denote the osculating conics to the curve $F=0$. Then:
\begin{itemize}
    \item[a)] All curves of the form $F \cdot C_i = 0$ are nearly free with exponents $(2,3,3)$.
    
    \item[b)] There exist exactly nine disjoint sets $G_\alpha \subset \{1,\dots,27\}$, each of size three, such that
    $$
    \bigcup_{\alpha} G_\alpha = \{1,\dots,27\},
    $$
    and:
    \begin{itemize}
        \item the curve $F \cdot C_i \cdot C_j \cdot C_k = 0$, for $\{i,j,k\} = G_\alpha$, is free with exponents $(3,5)$;
        \item for each distinct indices in the set $\{i,j\} \subset G_\alpha$, the curve $F \cdot C_i \cdot C_j = 0$ is free with exponents $(3,3)$;
        \item for distinct indices in the set $\{i,j\} \not\subset G_\alpha$, the curve $F \cdot C_i \cdot C_j = 0$ is nearly free with exponents $(3,4,4)$.
    \end{itemize}
\end{itemize}
\end{thm}
\begin{proof}
    The case $t=0$ was established in \cite{dimca}. The remaining cases follow by analogous arguments and are thus omitted.
\end{proof}

A natural question that arises at this point is whether harmonic cubics exhibit similar properties. It turns out that the answer is negative. Solving the equation $$\lambda^6-20\lambda^3\mu^3-8\mu^6=0$$ yields, among other values, $\frac{6\mu}{\lambda} \in \{-3(1-\sqrt{3}),-3(1+\sqrt{3})\}$. As verified using the {\tt Singular} software, for these values of $t$, the resulting curves exhibit the same behavior as in the case $t=-5$, which corresponds to a cubic curve that does not belong to either of the two previously discussed families. 

\begin{thm}
Let $F(x,y,z)=x^3+y^3+z^3+txyz$, and let $t\in\{-5,-3(1-\sqrt{3}),-3(1+\sqrt{3})\}$. Let $C_i$, for $i \in \{1,2,\ldots,27\}$, denote the osculating conics to the curve $F=0$. Then:
\begin{itemize}
    \item[a)] Each curve defined by $F \cdot C_i = 0$ is nearly free with exponents $(2,3,3)$.
    
    \item[b)] Each curve of the form $F \cdot C_i \cdot C_j = 0$, for distinct indices $i,j \in \{1,\ldots,27\}$, is nearly free with exponents $(3,4,4)$.
    
    \item[c)] Each curve of the form $F \cdot C_i \cdot C_j \cdot C_k = 0$, for distinct indices $i,j,k \in \{1,\ldots,27\}$, has exponents $(5,5,5)$.
\end{itemize}
\end{thm}

The reader's convenience, we provide {\tt Singular} programs that compute the exponents of the plane curves defined by the equations $F \cdot C_i=0$, $F \cdot C_i \cdot C_j=0$, and $F \cdot C_i \cdot C_j \cdot C_k=0$ for all values of $t$ considered in this section. These programs are available at \cite{github}. The reader is encouraged to use these files to verify our computations and compare the numerical output with the results stated in the paper.

\newpage

\vskip 0.5 cm
\bigskip
\noindent
Ewelina Nawara,
Department of Mathematics, University of the National Education Commission, Podchor\c a\.zych 2, 30-084 Krak\'ow, Poland,\\
\nopagebreak
\textit{E-mail address:} \texttt{ewelina.nawara@op.pl}


\begin{thebibliography}{99}\footnotesize

\bibitem{dolga} Artebani, M., Dolgachev, I.: {\it The Hesse Pencil of plane cubic curve.} Enseign. Math. 55 (2009), no. 3/4, 235--273

\bibitem{stud} Balay-Wilson, L., Brysiewicz, T.: {\it Points of ninth order on cubic curves}. Rose-Hulman Undergraduate Mathematics Journal 15 (2014), 1--22

\bibitem{Belta} Beltrametti C.M., Carletti E., Gallarati D., Bragadin G.M.,: {\it Lectures on Curves, Surfaces and Projective Varieties}. European Mathematical Society (EMS), Zürich, 2009, xvi+491 pp

\bibitem{Cay59} Cayley, A.: {\it On the Conic of Five-Pointic Contact at Any Point of a Plane Curve}, Philosophical Transactions of the Royal Society of London \textbf{149}, 371--400 (1859).

\bibitem{Cay65} Cayley, A.: {\it On the Sextactic Points of a Plane Curve}, Philosophical Transactions of the Royal Society of London \textbf{155},  545--578 (1865).

\bibitem{Coo} Coolidge J.L.: {\it A Treatise on Algebraic Plane Curves,} Clarendon Press, pp.xxiv+513 (1931).

\bibitem{DGPS}
Decker, W., Greuel, G.-M., Pfister, G., Sch{\"o}nemann, H.: 
\newblock {\sc Singular} {4-4-0} --- {A} computer algebra system for polynomial computations.
\newblock {https://www.singular.uni-kl.de} (2024).

\bibitem{dimca} Dimca A., Ilardi G., Malara G., Pokora P.: {\it Construction of free curves by adding osculating conics to a given cubic curve}, to appear in: Int. Math. Res. Not.,  \url{https://doi.org/10.1093/imrn/rnae273}
 
\bibitem{Halphen} Halphen G.-H.:{\it Oeuvres de G.H. Halphen}, Gauthier-Villars Tom 2, 198 (1918).

\bibitem{Hart} Hart A.S.:{\it On Nine-Point Contact of Cubic Curves,} Science 559--566 (1875).

\bibitem{MaMo19} Maugesten, P. A., Moe, T. K.: {\it The 
$2$-Hessian and sextactic points on plane algebraic curves}. Mathematica Scandinavica, \textbf{125}(1) (2019), 13--38

\bibitem{masaaki} Masaaki U., Thorbergsson G.:{\it Sextactic points on a simple closed curve,} Nagoya Math. J \textbf{167}, 55-94 (2002).

\bibitem{MT}
Moe, T. K., Toft, N. P. A.:{\it The Fermat curves and arrangements of lines and conics}, preprint: {\tt arXiv:2412.16993}

\bibitem{MZ}
Merta, \L{}., Zi\k{e}ba, M.:{\it Sextactic and type-9 points on the fermat cubic and associated objects}, J. Algebra \textbf{662} (2025), 502--513

\bibitem{SzeSzp}  Szemberg, T., Szpond, J.: {\it Sextactic points on the Fermat cubic curve and arrangements of conics}. J. Symb. Comput. \textbf{120}, (2024).


\bibitem{github}
Computations performed with {\tt Singular}: 
\url{https://github.com/EwelinaNaw/Sextactic-Points.git}




\end{thebibliography}
\end{document}